\documentclass[11pt]{amsart}

\usepackage{amsfonts, fancyhdr, qtree, amsmath, amscd, tipa, amssymb, hyperref, url, ifsym, amsthm, subfigure, enumitem, wasysym, tikz}
\usetikzlibrary{matrix, arrows}

\newtheorem{thm}{Theorem}[section]
\newtheorem{lem}[thm]{Lemma}

\newtheorem{prop}[thm]{Proposition}

\newtheorem{cor}[thm]{Corollary}

\theoremstyle{definition}
\newtheorem{df}[thm]{Definition}
\theoremstyle{remark}
\newtheorem{rem}[thm]{Remark}

\numberwithin{equation}{section}

\title{On Orbits of Order Ideals of Minuscule Posets II: Homomesy}
\author{David B Rush \and Kelvin Wang}
\address{Department of Mathematics, Massachusetts Institute of Technology, 77 Massachusetts Ave, Cambridge, MA 02139}
\email{dbr@mit.edu}
\address{Department of Mathematics, Harvard University, 1 Oxford St, Cambridge, MA 02138}
\email{kelvinwang@college.harvard.edu}
\date{\today}

\begin{document}

\begin{abstract}
The Fon-Der-Flaass action partitions the order ideals of a poset into disjoint orbits.  For a product of two chains, Propp and Roby observed --- across orbits --- the mean cardinality of the order ideals within an orbit to be invariant.  That this phenomenon, which they christened homomesy, extends to all minuscule posets is shown herein.  

Given a minuscule poset $P$, there exists a complex simple Lie algebra $\mathfrak{g}$ and a representation $V$ of $\mathfrak{g}$ such that the lattice of order ideals of $P$ coincides with the weight lattice of $V$.  For a weight $\mu$ with corresponding order ideal $I$, it is demonstrated that the behavior of the Weyl group simple reflections on $\mu$ not only uniquely determines $\mu$, but also encodes the cardinality of $I$.  After recourse to work of Rush and Shi mapping the anatomy of the lattice isomorphism, the upshot is a uniform proof that the cardinality statistic exhibits homomesy.  

A further application of these ideas shows that the statistic tracking the number of maximal elements in an order ideal is also homomesic, extending another result of Propp and Roby.  
\end{abstract}

\maketitle

\section{Introduction}

Whenever the minuscule posets --- a class of partially ordered sets comprising three infinite families and two exceptional cases --- are seen to satisfy a combinatorial property, a search for a uniform explanation invariably follows.  Recently, the minuscule posets have emerged at the front lines of inquiry into the Fon-Der-Flaass action, an action on order ideals of a poset introduced in its original form on hypergraphs by Duchet in 1974 \cite{Duchet}.  First to consider the action in a minuscule setting was Fon-Der-Flaass himself \cite{FDF}, who computed the order of the action in products of two chains, which constitute the first infinite family of minuscule posets.  Subsequent work of Stanley \cite{StanleyP} on products of two chains permitted Striker and Williams \cite{Striker} to note that the Fon-Der-Flaass action in the first two infinite families of minuscule posets exhibits the cyclic sieving phenomenon of Reiner, Stanton, and White \cite{Reiner}.  In collaboration with Shi, the first author then gave a uniform proof \cite{Rush} that cyclic sieving occurs in all minuscule posets.  

This sequel builds on the techniques of Rush--Shi \cite{Rush} to demonstrate that two statistics on order ideals --- the first associating to each order ideal its cardinality, and the second the cardinality of its antichain of maximal elements --- exhibit homomesy with respect to the the Fon-Der-Flaass action in minuscule posets.  Our immediate aim is to generalize the corresponding results of Propp and Roby \cite{Propp} on homomesy of the Fon-Der-Flaass action in products of two chains.  But in tracing homomesy as well as cyclic sieving back to origins in the representation theory from which the minuscule posets arise, we seek ultimately to underscore the consistency of the behavior of the Fon-Der-Flaass action in minuscule posets and again affirm that the approach inaugurated in \cite{Rush} exposes the Fon-Der-Flaass action to algebraic avenues of attack.  

Let $P$ be a poset, and let $J(P)$ be the set of order ideals of $P$, partially ordered by inclusion.  The Fon-Der-Flaass action $\Psi \colon J(P) \rightarrow J(P)$ maps an order ideal $I$ to the order ideal generated by the minimal elements of $P \setminus I$.  In other words, for all $I \in J(P)$, the maximal elements of $\Psi(I)$ and the minimal elements of $P \setminus I$ coincide.  

We may reframe this rule as governing the behavior on order ideals of local actions, which, following Striker and Williams \cite{Striker}, we refer to as \textit{toggles}.  For all $p \in P$ and $I \in J(P)$, \textit{toggling $I$ at $p$} yields the symmetric difference $I \triangle \lbrace p \rbrace$ if $I \triangle \lbrace p \rbrace \in J(P)$ and returns $I$ otherwise.  Let $t_p$ denote the action of toggling at $p$.   Then the local interpretation of the global description of $\Psi$ is that $t_p(I) = I \cup \lbrace p \rbrace$ if and only if $t_p(\Psi(I)) = I \setminus \lbrace p \rbrace$.

Note that each order ideal of $P$ is uniquely determined by its antichain of maximal elements $\lbrace p \in P : t_p(I) = I \setminus \lbrace p \rbrace \rbrace$, so we may interpret $\Psi$ as an action either on order ideals or on antichains of $P$.  (The latter is the original perspective of Fon-Der-Flaass \cite{FDF}; the advantages of the former are apparent when representation theory enters the picture.)  

Let $\mathfrak{g}$ be a complex simple Lie algebra with Weyl group $W$ and weight lattice $\Lambda$.  Let $\lambda \in \Lambda$ be dominant, and let $V^{\lambda}$ be the irreducible $\mathfrak{g}$-representation with highest weight $\lambda$.  If $W$ acts transitively on the weights of $V^{\lambda}$, we say that $V^{\lambda}$ is a \textit{minuscule $\mathfrak{g}$-representation} with \textit{minuscule weight} $\lambda$.  In this case, the restriction to $W \lambda$ of the partial order on $\Lambda$ opposite to the root order is a distributive lattice, and the poset of join-irreducible elements $P_{\lambda}$ is the \textit{minuscule poset} for $V^{\lambda}$.

The \textit{minuscule heap} for $V^{\lambda}$ appends to $P_{\lambda}$ the assignment of a simple root $\alpha(p)$ of $\mathfrak{g}$ to each element $p \in P_{\lambda}$.  The heap labeling of a minuscule poset, pioneered by Stembridge in \cite{Stembridge} and \cite{Stembridge2}, is an essential ingredient of our proofs in \cite{Rush} and in the present article.  We postpone a discussion of its significance until after the statement of our main theorems.  

\begin{df} \label{homomesy}
Let $\mathcal{S}$ be a finite set, and let $\tau \colon \mathcal{S} \rightarrow \mathcal{S}$ be an action on $\mathcal{S}$ of order $m$.  A function $f \colon \mathcal{S} \rightarrow \mathbb{R}$ \textit{exhibits homomesy} with respect to $\tau$ if there exists a constant $c \in \mathbb{R}$ such that, for all $x \in S$, the following equality holds:
\[
\frac{1}{m} \sum_{i=0}^{m-1} f(\tau^i(x)) = c.
\]

In this case, $f$ is \textit{$c$-mesic} with respect to $\tau$.  
\end{df}

\begin{thm} \label{mainorder}
Let $V$ be a minuscule $\mathfrak{g}$-representation with minuscule weight $\lambda$ and minuscule heap $P_{\lambda}$.  Let $\alpha$ be a simple root of $\mathfrak{g}$ with corresponding fundamental weight $\omega$, and let $P_{\lambda}^{\alpha} \subset P_{\lambda}$ be the set of elements of $P_{\lambda}$ labeled by $\alpha$.  Let $f^{\alpha} \colon J(P_{\lambda}) \rightarrow \mathbb{R}$ be defined by $I \mapsto |I \cap P_{\lambda}^{\alpha}|$.  Then, with respect to the Fon-Der-Flaass action, $f^{\alpha}$ is $c$-mesic with $c = 2 \frac{(\lambda, \omega)}{(\alpha,\alpha)}$.
\end{thm}

\begin{cor} \label{cororder}
Let $V$ be a minuscule $\mathfrak{g}$-representation with minuscule weight $\lambda$ and minuscule heap $P_{\lambda}$.  Suppose that $\mathfrak{g}$ is simply laced.  Let $\rho$ denote the half-sum of the positive roots of $\mathfrak{g}$, and let $\Omega$ denote the common length of the roots.  Let $f \colon J(P_{\lambda}) \rightarrow \mathbb{R}$ be defined by $I \mapsto |I|$.  Then, with respect to the Fon-Der-Flaass action, $f$ is $c$-mesic with $c = 2 \frac{(\lambda, \rho)}{\Omega^2}$.  
\end{cor}

\begin{thm} \label{mainanti}
Let $V$ be a minuscule $\mathfrak{g}$-representation with minuscule weight $\lambda$ and minuscule heap $P_{\lambda}$.  Suppose that $\mathfrak{g}$ is simply laced.  Let $g \colon J(P_{\lambda}) \rightarrow \mathbb{R}$ be defined by $I \mapsto |\lbrace p \in P_{\lambda} : t_p(I) = I \setminus \lbrace p \rbrace \rbrace|$.  Then, with respect to the Fon-Der-Flaass action, $g$ is $c$-mesic with $c = 2 \frac{(\lambda, \lambda)}{\Omega^2}$.  
\end{thm}

\begin{rem}
Given a minuscule poset $P$, there exists a simply laced $\mathfrak{g}$ for which $P$ arises as the minuscule poset of a minuscule $\mathfrak{g}$-representation.  In other words, the multiply laced Lie algebras yield redundant minuscule posets.  Hence Corollary~\ref{cororder} and Theorem~\ref{mainanti} imply that order ideal cardinality and antichain cardinality, respectively, are homomesic with respect to $\Psi$ in all minuscule posets.  
\end{rem}

By construction, the lattice of order ideals $J(P_{\lambda})$ and the weight lattice $W \lambda$ are isomorphic as posets.  The primary purpose of the heap labeling of $P_{\lambda}$ is to identify each covering relation in $P_{\lambda}$ with a counterpart in $W \lambda$ and thereby generate an isomorphism explicitly.  If $I \lessdot I'$ is a covering relation in $J(P_{\lambda})$, then there exists an element $p \in P_{\lambda}$ such that $I' \setminus I = \lbrace p \rbrace$, and toggling at $p$ transitions back and forth between $I$ and $I'$.  Similarly, if $\mu \lessdot \mu'$ is a covering relation in $W \lambda$, then there exists a simple root $\alpha \in \Lambda$ such that $\mu - \mu' = \alpha$, and the simple reflection $s_{\alpha} \in W$ interchanges $\mu$ and $\mu'$.  The heap label of each element $p \in P_{\lambda}$ is the simple root $\alpha(p)$ that stipulates the simple reflection to correspond to toggling at $p$.  

The association of simple reflections to toggles induces a bijection between saturated chains of $J(P_{\lambda})$ and $W \lambda$ that results in an isomorphism between the two lattices.  To wit, if $I$ is an order ideal of $P_{\lambda}$, and $(p_1, p_2, \ldots, p_{\ell})$ is a linear extension of $I$, then $I = t_{p_{\ell}} t_{p_{\ell-1}} \cdots t_{p_1} (\varnothing)$, and we define \[\phi(I) := s_{\alpha(p_{\ell})} s_{\alpha(p_{\ell -1 })} \cdots s_{\alpha(p_1)} (\lambda).\]  

That $\phi \colon J(P_{\lambda}) \rightarrow W \lambda$ is a well-defined isomorphism is due to Stembridge \cite{Stembridge}.  In Rush--Shi \cite{Rush}, it is shown that the heap labels $\lbrace \alpha(p) \rbrace_{p \in P}$ indeed indicate the corresponding covering relations, for the action of each simple reflection $s_{\alpha}$ is parallel under $\phi$ to that of the sequence of toggles at all elements of $P_{\lambda}$ to which the label $\alpha$ is affixed.  

\begin{lem}[Rush--Shi \cite{Rush}] \label{equiviso}
Let $V$ be a minuscule $\mathfrak{g}$-representation with minuscule weight $\lambda$, and let $P_{\lambda}$ be the minuscule heap for $V$.  Let $\alpha$ be a simple root of $\mathfrak{g}$.  Let $t_{\alpha} := \prod_{p \in P_{\lambda}^{\alpha}} t_p$.  Then the following diagram is commutative.  
\[\renewcommand{\arraystretch}{1.0}
\begin{array}[c]{ccc}
J(P_{\lambda}) & \stackrel{\phi}{\rightarrow} & W \lambda \\
\downarrow \scriptstyle{t_{\alpha}} && \downarrow \scriptstyle{s_{\alpha}} \\
J(P_{\lambda}) & \stackrel{\phi}{\rightarrow} & W \lambda
\end{array}
\]
\end{lem}

If $\mathfrak{g} = \mathfrak{sl}_n$, then $W \cong \mathfrak{S}_n$, and $\Lambda$ is isomorphic to a quotient of $\mathbb{Z}^n$.  Choosing the set of simple roots $\Delta := \lbrace \alpha_1, \alpha_2, \ldots, \alpha_{n-1} \rbrace$ by $\alpha_i := e_{i+1} - e_i$ for all $1 \leq i \leq n-1$, where $e_j$ denotes the image in $\Lambda$ of the $j^{\text{th}}$ standard basis vector in $\mathbb{Z}^n$ for all $1 \leq j \leq n$, we see that the fundamental weights take the form $\omega_i = e_{i+1} + e_{i+2} + \cdots + e_n$, and we observe, for dominant $\lambda \in \Lambda$, that $V^{\lambda}$ is minuscule precisely when $\lambda$ is fundamental.  

Set $\lambda := \omega_{n-i}$.  In this case, the minuscule representation is $V^{\lambda} = \wedge^{i} (\mathbb{C}^n)$; the weight lattice is $W \lambda = \lbrace e_{j_1} + e_{j_2} + \cdots + e_{j_i} : 1 \leq j_1 < j_2 < \cdots < j_i \leq n \rbrace$, and the minuscule poset is $P_{\lambda} = [n-i] \times [i]$, which may be realized in $\mathbb{Z}^2$ as the set $\lbrace (b-a, b+a): 1 \leq a \leq n-i, 1 \leq b \leq i \rbrace$, partially ordered by the relations $\lbrace (b-a, b+a) \leq (b'-a', b'+a') : a \leq a', b \leq b' \rbrace$.  For $1 \leq \ell \leq n-1$, Propp and Roby \cite {Propp}, following Striker and Williams \cite{Striker}, consider the elements with $x$-coordinate $\ell - n + i$ to belong to the $\ell^{\text{th}}$ \textit{column} of $P_{\lambda}$.  As it happens, these are exactly the elements of $P_{\lambda}$ labeled by the simple root $\alpha_{\ell}$ in the minuscule heap for $\wedge^i (\mathbb{C}^n)$.  

In \cite{Propp}, Propp and Roby undertake a systematic study of cyclic actions on finite sets for which there exist statistics that exhibit homomesy.  For their paradigmatic example of the Fon-Der-Flaass action in products of two chains, they find the following statistics on an order ideal to be homomesic: the cardinality of each of its columns, its total cardinality, and the cardinality of its corresponding antichain.  In this article, we present a generalization of these results that extends to all minuscule posets and, by subsuming the notion of ``columns'' in that of sets of identically labeled heap elements, illuminates the representation-theoretic underpinnings of the Propp--Roby choices of statistic. 

The proofs are subtle, but they are very simple.  Let $\lbrace \alpha_1, \alpha_2, \ldots, \alpha_t \rbrace$ be a choice of set of simple roots for $\mathfrak{g}$, and let $\lbrace \alpha_1^{\vee}, \alpha_2^{\vee}, \ldots, \alpha_t^{\vee} \rbrace$ be the corresponding set of simple coroots.  We obtain Theorems~\ref{mainorder} and \ref{mainanti} as codas to a dogged examination of the inner products $(\mu, \alpha_i^{\vee})$ for weights $\mu \in W \lambda$ and $1 \leq i \leq t$.  

From the relation $s_{\alpha_i}(\mu) = \mu - (\mu, \alpha_i^{\vee}) \alpha_i$, we see in view of Lemma~\ref{equiviso} that $(\mu, \alpha_i^{\vee})$ reflects the number of heap elements labeled by $\alpha_i$ that may be toggled in or out of the order ideal $I$ for which $\phi(I) = \mu$.  Thus, the local characterization of the Fon-Der-Flaass action entails that $(\phi(I), \alpha_i^{\vee})$ varies predictably as $I$ ranges over an orbit under $\Psi$.  Because the simple roots form a basis for $\Lambda$, each order ideal $I$ is uniquely determined by the values $(\phi(I), \alpha_i^{\vee})$ for $1 \leq i \leq t$.  To deduce Theorem~\ref{mainorder}, all we require is a systematic method to extract the cardinalities $f^{\alpha_i}(I)$ from these inner products, which we herein contrive.  The computations that lead to the proof of Theorem~\ref{mainanti} are similar in spirit, only more involved.  

The rest of this article is organized as follows.  In section 2, we review the background on minuscule posets and minuscule heaps, and we restate the Rush--Shi \cite{Rush} lemma.  In section 3, we canvass the inner products $(\mu, \alpha_i^{\vee})$ and prove Theorems~\ref{mainorder} and \ref{mainanti}.   

\section{Minuscule Posets}

In this section, we introduce the minuscule posets and their labeled incarnations, the minuscule heaps.  In accordance with the Rush--Shi \cite{Rush} lemma, the labeling of a minuscule heap encapsulates the relationship between the covering relations in its lattice of order ideals and those in the corresponding weight lattice.  We start with basics on Lie algebras, root systems, and weights, following Kirillov \cite{Kirillov}.

Let $\mathfrak{g}$ be a complex simple Lie algebra, and let $\mathfrak{h}$ be a choice of Cartan subalgebra.  Note that the restriction to $\mathfrak{h}$ of the Killing form on $\mathfrak{g}$ is non-degenerate (cf. Kirillov \cite{Kirillov}, Theorem 6.38), so it induces a symmetric bilinear form on $\mathfrak{h}^*$, which we denote by $(\cdot, \cdot)$.  

Let $R \subset \mathfrak{h}^*$ be the roots of $\mathfrak{g}$.  Then $R$ spans $\mathfrak{h}^*$ as a complex vector space (cf. Kirillov \cite{Kirillov}, Theorem 6.44), and we denote by $\mathfrak{h}_{\mathbb{R}}^* \subset \mathfrak{h}^*$ the real vector space generated by $R$.  The restriction of $(\cdot, \cdot)$ to $\mathfrak{h}_{\mathbb{R}}^*$ is positive-definite (cf. Kirillov \cite{Kirillov}, Theorem 6.45), so $(\cdot, \cdot)$ is an inner product on $\mathfrak{h}_{\mathbb{R}}^*$.  From Theorem 7.3 of Kirillov \cite{Kirillov}, it follows that $R$ is a reduced root system for the vector space $\mathfrak{h}_{\mathbb{R}}^*$ equipped with the inner product $(\cdot, \cdot)$.  

For all $\alpha \in R$, let $\alpha^{\vee} := 2 \frac{\alpha}{(\alpha,\alpha)}$ be the coroot associated to $\alpha$.  Let $\Pi = \lbrace \alpha_1, \alpha_2, \ldots, \alpha_t \rbrace$ be a choice of set of simple roots for $R$, and write $\Pi^{\vee} := \lbrace \alpha_1^{\vee}, \alpha_2^{\vee}, \ldots, \alpha_t^{\vee} \rbrace$ for the corresponding set of simple coroots.  

\begin{prop}[Kirillov \cite{Kirillov}, Theorem 7.16] \label{rootbasis}
The set of simple roots $\Pi$ constitutes a basis for $\mathfrak{h}_{\mathbb{R}}^*$.  
\end{prop}

\begin{cor} \label{cobasis}
The set of simple coroots $\Pi^{\vee}$ constitutes a basis for $\mathfrak{h}_{\mathbb{R}}^*$.  
\end{cor}

We refer to the lattices $\Phi$ and $\Phi^{\vee}$ in $\mathfrak{h}_{\mathbb{R}}^*$ generated over $\mathbb{Z}$ by $\Pi$ and $\Pi^{\vee}$ as the \textit{root} and \textit{coroot lattices}, respectively.  The \textit{weight lattice} $\Lambda$, which comprises the weights of $\mathfrak{g}$, is the dual lattice to the coroot lattice.  

\begin{df} \label{weight}
A functional $\lambda \in \mathfrak{h}_{\mathbb{R}}^*$ is a \textit{weight} of $\mathfrak{g}$ if $(\lambda, \alpha_i^{\vee}) \in \mathbb{Z}$ for all $1 \leq i \leq t$.  
\end{df}

\begin{df} \label{dominant}
A weight $\lambda \in \Lambda$ is \textit{dominant} if $(\lambda, \alpha_i^{\vee}) \geq 0$ for all $1 \leq i \leq t$.  
\end{df}

We write $\Lambda^+ \subset \Lambda$ for the subset of dominant weights.  Since the weight lattice is dual to $\Phi^{\vee}$, the biorthogonal basis to $\Pi^{\vee}$ is a distinguished subset of $\Lambda^+$ that forms a $\mathbb{Z}$-basis for $\Lambda$.  We call the weights in this basis \textit{fundamental}.  

\begin{df} \label{fundamental}
The \textit{fundamental weights} $\omega_1, \omega_2, \ldots, \omega_t$ are defined by the relations $(\omega_i, \alpha_j^{\vee}) = \delta_{ij}$ for $1 \leq i, j \leq t$, where $\delta_{ij}$ denotes the Kronecker delta.     
\end{df}

The dominant weights in $\Lambda$ index the finite-dimensional irreducible representations of $\mathfrak{g}$.  

\begin{thm}[Kirillov \cite{Kirillov}, Corollary 8.24] \label{weightrep}
For all $\lambda \in \Lambda^+$, there exists a finite-dimensional irreducible representation $V^{\lambda}$ of $\mathfrak{g}$ with highest weight $\lambda$.  Furthermore, the map $\lambda \mapsto [V^{\lambda}]$ defines a bijection between $\Lambda^+$ and the set of isomorphism classes of finite-dimensional irreducible $\mathfrak{g}$-representations.    
\end{thm}

Let $V$ be a finite-dimensional irreducible representation of $\mathfrak{g}$.  For all $\mu \in \Lambda$, the weight $\mu$ of $\mathfrak{g}$ is a \textit{weight} of $V$ if the weight space associated to $\mu$, namely, \[\lbrace v \in V : hv = \mu(h)v \text{  } \forall h \in \mathfrak{h} \rbrace,\] is nonzero.  For all $\lambda \in \Lambda^+$, we write $\Lambda_{\lambda} \subset \Lambda$ for the (finite) subset comprising the weights of $V^{\lambda}$.  Note that if $\lambda$ is the unique dominant weight of $\mathfrak{g}$ for which $V \cong V^{\lambda}$, then the weights of $V$ and $V^{\lambda}$ coincide.  

For all roots $\alpha \in R$, let the \textit{reflection} associated to $\alpha$ be the orthogonal involution on $\mathfrak{h}_{\mathbb{R}}^*$ given by $\lambda \mapsto \lambda - (\lambda, \alpha^{\vee}) \alpha$.  Recall that the \textit{Weyl group} $W$ of $\mathfrak{g}$ is the subgroup of $O(\mathfrak{h}_{\mathbb{R}}^*)$ generated by the set of \textit{simple reflections} $\lbrace s_{i}\rbrace_{i=1}^t$, where $s_i := s_{\alpha_i}$ is the reflection associated to the simple root $\alpha_i$ for all $1 \leq i \leq t$.  It is easy to see that $W$ preserves $\Phi$, $\Phi^{\vee}$, and $\Lambda$.  As it happens, it preserves $\Lambda_{\lambda}$ as well, so the action of $W$ on $\Lambda$ restricts to an action on $\Lambda_{\lambda}$ (cf. Kirillov \cite{Kirillov}, Theorem 8.8).  

At last we come to the definition of a minuscule representation, which underlies that of a minuscule poset.  

\begin{df} \label{minurep}
Let $V$ be a finite-dimensional irreducible representation of $\mathfrak{g}$, and let $\lambda$ be the unique dominant weight of $\mathfrak{g}$ for which $V \cong V^{\lambda}$.  Then $V$ is \textit{minuscule} with \textit{minuscule weight} $\lambda$ if the action of $W$ on $\Lambda_{\lambda}$ is transitive.  
\end{df}

\begin{rem}
If $\lambda$ is a dominant weight of $\mathfrak{g}$, and the $\mathfrak{g}$-representation $V^{\lambda}$ is minuscule, then $\lambda$ is fundamental.  The converse holds for $\mathfrak{g} = \mathfrak{sl}_n$, but it does not hold in general.  
\end{rem}

To see how a minuscule poset arises from a minuscule representation $V$ with minuscule weight $\lambda$, we require a partial order on the weights $\Lambda_{\lambda}$.  One choice, and that taken in Rush--Shi \cite{Rush}, is the restriction to $\Lambda_{\lambda}$ of the \textit{root order} on $\Lambda$, viz., the transitive closure of the relations $\mu \lessdot \nu$ for $\mu, \nu \in \Lambda$ satisfying $\nu - \mu \in \Pi$.  Here, however, we follow Proctor \cite{Proctor}, and opt for the opposite order on $\Lambda_{\lambda}$.  The considerations that motivate us to make this change are technical, but chief among them is our conviction that the empty order ideal of the minuscule poset for $V^{\lambda}$ ought correspond to $\lambda$ in the order-preserving bijection between order ideals and weights we ultimately seek to define.  Thus, in what follows, we assume $\Lambda_{\lambda}$ is endowed with the partial order opposite to that inherited from the root order, so that $\lambda \in \Lambda_{\lambda}$ is the unique minimal weight. 

\begin{rem}
Let $w_0$ be the longest element of $W$.  Since $-w_0 \colon \Lambda \rightarrow \Lambda$ preserves $\Pi$, it follows that $\nu \lessdot \mu$ is a covering relation in $\Lambda_{\lambda}$ if and only if $w_0 \mu \lessdot w_0 \nu$ is as well.  Hence $w_0$ defines an order-reversing involution on $\Lambda_{\lambda}$ that renders the choice between the partial order we adopt and its inverse cosmetic in character.  
\end{rem}

\begin{prop}
Let $\lambda \in \Lambda^+$ such that the $\mathfrak{g}$-representation $V^{\lambda}$ is minuscule.  Then $\Lambda_{\lambda}$ is a distributive lattice.  
\end{prop}

\begin{proof}
See Proctor \cite{Proctor}, Propositions 3.2 and 4.1.  For a uniform proof, see Stembridge \cite{Stembridge}, Theorems 6.1 and 7.1.  
\end{proof}

\begin{df}
Let $V$ be a minuscule representation of $\mathfrak{g}$ with minuscule weight $\lambda$.  The restriction of the partial order on $\Lambda_{\lambda}$ to its join-irreducible elements is the \textit{minuscule poset} for $V$, which we denote by $P_{\lambda}$.  
\end{df}

\begin{rem}
If $L$ is any distributive lattice, and $P$ is its poset of join-irreducible elements, then $J(P) \cong L$ (cf. Stanley \cite{ec1}, Proposition 3.4.2).  Thus, if $V$ is a minuscule representation of $\mathfrak{g}$ with minuscule weight $\lambda$, and $P_{\lambda}$ is its minuscule poset, then $J(P_{\lambda})$ is a distributive lattice isomorphic to $\Lambda_{\lambda}$.  
\end{rem}

\begin{df}
Let $P$ be a poset.  Then $P$ is \textit{minuscule} if there exists a complex simple Lie algebra $\mathfrak{g}$ and a dominant weight $\lambda$ of $\mathfrak{g}$ for which the $\mathfrak{g}$-representation $V^{\lambda}$ is minuscule and $P \cong P_{\lambda}$.  
\end{df}

Suppose that $\lambda$ is a dominant weight of $\mathfrak{g}$ for which $V^{\lambda}$ is minuscule.  For all $\mu \in \Lambda_{\lambda}$, we denote by $\Lambda_{\lambda}^{\mu}$ the restriction of the partial order on $\Lambda_{\lambda}$ to the set $\lbrace \nu \in \Lambda_{\lambda} : \nu \leq \mu \rbrace$.  Since $\Lambda_{\lambda}$ is a distributive lattice, we obtain a family of distributive lattices indexed by $\Lambda_{\lambda}$.    

Concomitant with $\lbrace \Lambda_{\lambda}^{\mu} \rbrace_{\mu \in \Lambda_{\lambda}}$ is a family of labeled posets $\lbrace P_{\lambda, \mu} \rbrace_{\mu \in \Lambda_{\lambda}}$, which we refer to as \textit{heaps}, constructed so that the labeled linear extensions of $P_{\lambda, \mu}$ catalogue the maximal chains of $\Lambda_{\lambda}^{\mu}$.  Fixing $\mu$, we see that for all $\nu \in \Lambda_{\lambda}^{\mu}$, the maximal chains of $\Lambda_{\lambda}^{\nu}$ are precisely the saturated chains of $\Lambda_{\lambda}^{\mu}$ originating at $\lambda$ and terminating at $\nu$, so the correspondence between labeled linear extensions of $P_{\lambda, \nu}$ and maximal chains of $\Lambda_{\lambda}^{\nu}$ embeds $P_{\lambda, \nu}$ as an order ideal of $P_{\lambda, \mu}$.  In unison, these correspondences determine that the map $J(P_{\lambda, \mu}) \rightarrow \Lambda_{\lambda}^{\mu}$ by $P_{\lambda, \nu} \mapsto \nu$ into which they combine is an isomorphism.  

The \textit{raison d'\^etre} of our excursion into heaps is the realization of an explicit isomorphism $J(P_{\lambda}) \cong \Lambda_{\lambda}$.  We designate the heap $P_{\lambda, w_0 \lambda}$, which accomplishes this task, the \textit{minuscule heap} for $V^{\lambda}$.  Note that $P_{\lambda, w_0 \lambda}$ and $\Lambda_{\lambda}^{w_0 \lambda}$ coincide as posets with $P_{\lambda}$ and $\Lambda_\lambda$, respectively, so the map $J(P_{\lambda, w_0 \lambda}) \rightarrow \Lambda_{\lambda}^{w_0 \lambda}$ by $P_{\lambda, \mu} \mapsto \mu$ is indeed an isomorphism $J(P_{\lambda}) \xrightarrow{\sim} \Lambda_{\lambda}$.  

In Rush--Shi \cite{Rush}, we discussed heaps in the context of Bruhat posets, following Stembridge \cite{Stembridge}.  Here we wish to emphasize the representation-theoretic aspects of the story, so we hew more closely to the presentation of Stembridge \cite{Stembridge2}.  

\begin{prop}[Bourbaki \cite{Bourbaki}, Exercise VI.I.24] \label{bourbaki}   
Let $\mu \in \Lambda_{\lambda}$.  Then $(\mu, \alpha_i^{\vee}) \in \lbrace -1, 0, 1 \rbrace$ for all $1 \leq i \leq t$.  
\end{prop}

\begin{prop} \label{cover}
Let $\mu \in \Lambda_{\lambda}$.  If $\mu \lessdot \mu - \alpha_i$ is a covering relation, then $s_i (\mu) = \mu - \alpha_i$.
\end{prop}

\begin{proof}
Suppose $\mu \lessdot \mu - \alpha_i$ is a covering relation.  Since $(\mu, \alpha_i^{\vee}) - (\mu - \alpha_i, \alpha_i^{\vee}) = (\alpha_i, \alpha_i^{\vee}) = 2$, we see, in view of Proposition~\ref{bourbaki}, that $(\mu, \alpha_i^{\vee}) = 1$.  Hence $s_i (\mu) = \mu - \alpha_i$, as desired.
\end{proof}

\begin{df} \label{lamminus}
Let $w \in W$.  Then $w$ is \textit{$\lambda$-minuscule} if there exists a reduced word $w = s_{i_{\ell}} s_{i_{\ell-1}} \cdots s_{i_1}$ such that \[\lambda \lessdot s_{i_1} \lambda \lessdot \cdots \lessdot (s_{i_{\ell}} s_{i_{\ell-1}} \cdots s_{i_1}) \lambda = w \lambda\] is a saturated chain in $\Lambda_{\lambda}$.  
\end{df}

\begin{prop}[Stembridge \cite{Stembridge2}, Proposition 2.1] \label{everyred}
Let $w \in W$.  If $w$ is $\lambda$-minuscule and $w = s_{i_{\ell}} s_{i_{\ell-1}} \cdots s_{i_1}$ is a reduced word, then \[\lambda \lessdot s_{i_1} \lambda \lessdot \cdots \lessdot (s_{i_{\ell}} s_{i_{\ell-1}} \cdots s_{i_1}) \lambda = w \lambda\] is a saturated chain in $\Lambda_{\lambda}$.  
\end{prop}

\begin{prop} \label{everysat}
Let $\mu \in \Lambda_{\lambda}$.  Then there exists a unique $\lambda$-minuscule element $w \in W$ such that $w \lambda = \mu$.  Furthermore, if \[\lambda \lessdot \lambda - \alpha_{i_1} \lessdot \cdots \lessdot \lambda - \alpha_{i_1} - \alpha_{i_2} - \cdots - \alpha_{i_{\ell}} = \mu\] is a saturated chain in $\Lambda_{\lambda}$, then $s_{i_{\ell}} s_{i_{\ell-1}} \cdots s_{i_1}$ is a reduced word for $w$.  
\end{prop}

\begin{proof}
Let \[\lambda \lessdot \lambda - \alpha_{i_1} \lessdot \cdots \lessdot \lambda - \alpha_{i_1} - \alpha_{i_2} - \cdots - \alpha_{i_{\ell}} = \mu\] be a saturated chain in $\Lambda_{\lambda}$, and take $w := s_{i_{\ell}} s_{i_{\ell-1}} \cdots s_{i_1}$.  From Proposition~\ref{cover}, it follows that $w \lambda = \mu$.  The proof that $w$ is the unique $\lambda$-minuscule element such that $w \lambda = \mu$, and that all saturated chains originating at $\lambda$ and terminating at $\mu$ correspond to reduced words for $w$ is technical and not sufficiently pertinent to our purposes here to merit its inclusion.  However, we note that it emerges without altogether much effort from Proposition 4.1 in Proctor \cite{Proctor} and Theorems 6.1 and 7.1 of Stembridge \cite{Stembridge}.  
\end{proof}

\begin{df} \label{heap}
Let $w \in W$ be $\lambda$-minuscule, and let $w = s_{i_{\ell}} s_{i_{\ell-1}} \cdots s_{i_1}$ be a reduced word.  The \textit{heap} $P_{\lambda, (i_1, i_2, \ldots, i_{\ell})}$ associated to $s_{i_{\ell}} s_{i_{\ell -1}} \cdots s_{i_1}$ is the labeled set $\lbrace 1, 2, \ldots, \ell \rbrace$, where $i_j$ is the \textit{label} of the element $j$ for all $1 \leq j \leq \ell$, equipped with the partial order arising as the transitive closure of the relations $j < j'$ for all $1 \leq j < j' \leq \ell$ for which $s_{i_j}$ and $s_{i_{j'}}$ do not commute.  
\end{df}

\begin{prop} \label{linext}
Let $w \in W$ be $\lambda$-minuscule, and let $w = s_{i_{\ell}} s_{i_{\ell-1}} \cdots s_{i_1}$ be a reduced word.  Let $\mathcal{L}\left(P_{\lambda, (i_1, i_2, \ldots, i_{\ell})}\right) := \lbrace \mathcal{A} : \mathcal{A}_1 \lessdot \mathcal{A}_2 \lessdot \cdots \lessdot \mathcal{A}_{\ell} \rbrace$ be the set of linear extensions of the heap $P_{\lambda, (i_1, i_2, \ldots, i_{\ell})}$.  For all $\mathcal{A} \in \mathcal{L}$, let $s(A) := s_{i_{\mathcal{A}_{\ell}}} s_{i_{\mathcal{A}_{\ell-1}}} \cdots s_{i_{\mathcal{A}_1}}$.  Then $\lbrace s(A) \rbrace_{\mathcal{A} \in \mathcal{L}\left(P_{\lambda, (i_1, i_2, \ldots, i_{\ell})}\right)}$ is the set of reduced words for $w$ in $W$.  
\end{prop}

\begin{proof}
After translating from the setting of Stembridge in \cite{Stembridge2} to that in \cite{Stembridge}, we see that this follows from Proposition 1.2 of Stembridge \cite{Stembridge}.  
\end{proof}

Proposition~\ref{linext} tells us that the set of linear extensions of the heap associated to a reduced word for $w$ is independent of the choice of reduced word.  This suggests that the heaps associated to reduced words for $w$ are all pairwise isomorphic.  

\begin{prop}[Stembridge \cite{Stembridge2}, Proposition 2.1] \label{fullcomm}
Let $w \in W$ be $\lambda$-minuscule, and let $s_{i_{\ell}} s_{i_{\ell-1}} \cdots s_{i_1}$ and $s_{i'_{\ell}} s_{i'_{\ell-1}} \cdots s_{i'_1}$ be two reduced words for $w$.  Then there exists a sequence of commuting braid relations (viz., relations of the form $s_p s_q = s_q s_p$ for commuting simple reflections $s_p, s_q$) exchanging $s_{i_{\ell}} s_{i_{\ell-1}} \cdots s_{i_1}$ and $s_{i'_{\ell}} s_{i'_{\ell-1}} \cdots s_{i'_1}$.  
\end{prop}

\begin{rem}
Proposition~\ref{fullcomm} amounts to saying that $\lambda$-minuscule elements of $W$ are \textit{fully commutative} in the terminology of Stembridge \cite{Stembridge}.  
\end{rem}

\begin{prop} \label{heapiso}
Let $w \in W$ be $\lambda$-minuscule, and let $s_{i_{\ell}} s_{i_{\ell-1}} \cdots s_{i_1}$ and $s_{i'_{\ell}} s_{i'_{\ell-1}} \cdots s_{i'_1}$ be two reduced words for $w$.  Then there exists a permutation $\sigma \in \mathfrak{S}_{\ell}$ such that $\sigma \colon \lbrace 1, 2, \ldots, \ell \rbrace \rightarrow \lbrace 1, 2, \ldots, \ell \rbrace$ defines an isomorphism of heaps $P_{\lambda, (i_1, i_2, \ldots, i_{\ell})} \xrightarrow{\sim} P_{\lambda, (i'_1, i'_2, \ldots, i'_{\ell})}$.  
\end{prop}

\begin{proof}
It suffices to show the result under the assumption that there exists a commuting braid relation transforming $s_{i_{\ell}} s_{i_{\ell-1}} \cdots s_{i_1}$ into $s_{i'_{\ell}} s_{i'_{\ell-1}} \cdots s_{i'_1}$.  In this case, there exists $1 \leq k < \ell$ such that $i'_j$ agrees with $i_j$ for all $j \in \lbrace 1, 2, \ldots, \ell \rbrace \setminus \lbrace k, k+1 \rbrace$, and $i'_k = i_{k+1}$ and $i'_{k+1} = i_k$.  Let $\sigma \in \mathfrak{S}_{\ell}$ be the transposition exchanging $k$ and $k+1$.  Then $\sigma \colon P_{\lambda, (i_1, i_2, \ldots, i_{\ell})} \rightarrow P_{\lambda, (i'_1, i'_2, \ldots, i'_{\ell})}$ is a bijection of labeled sets, and it is an isomorphism of heaps because in both heaps the elements $k$ and $k+1$ are incomparable.  
\end{proof}

Since there is only one isomorphism class of heaps associated to reduced words for $w$, we refer to any heap associated to a reduced word for $w$ as the heap associated to $w$, and we denote it by $P_{\lambda, w}$.  For all $\mu \in \Lambda_{\lambda}$, we write $P_{\lambda, \mu}$ for the heap associated to the unique $\lambda$-minuscule element $w$ for which $w \lambda = \mu$.  It follows from Proposition~\ref{linext} (reinterpreted via Propositions~\ref{everyred} and \ref{everysat}) that the linear extensions of $P_{\lambda, \mu}$ catalogue the saturated chains originating at $\lambda$ and terminating at $\mu$.  To key to the description of the structure of $P_{\lambda, \mu}$ is the observation that the same relationship holds between the linear extensions of $P_{\lambda, \nu}$ and the saturated chains originating at $\lambda$ and terminating at $\nu$ for all $\nu \leq \mu$.

\begin{thm} \label{structure}
Let $\mu \in \Lambda_{\lambda}$.  Given an order ideal $I \in J(P_{\lambda, \mu})$, let $\mathcal{A}_I$ be a linear extension of $I$, and set $\phi(I) := s(\mathcal{A}_I) \lambda$.  Then $\phi$ defines an isomorphism of posets $J(P_{\lambda, \mu}) \xrightarrow{\sim} \Lambda_{\lambda}^{\mu}$.  
\end{thm}

\begin{proof}
Let $\nu \in \Lambda_{\lambda}^{\mu}$.  Let \[\lambda \lessdot  s_{i_1} \lambda \lessdot \cdots \lessdot s_{i_k} s_{i_{k-1}} \cdots s_{i_1} \lambda = \nu\] be a saturated chain in $\Lambda_{\lambda}^{\mu}$ that extends to a maximal chain \[\lambda \lessdot s_{i_1} \lambda \lessdot \cdots \lessdot s_{i_{\ell}} s_{i_{\ell-1}} \cdots s_{i_1} \lambda = \mu.\]  Since the heap $P_{\lambda, (i_1, i_2, \ldots, i_k)}$ is a labeled order ideal of the heap $P_{\lambda, (i_1, i_2, \ldots, i_{\ell})}$, it follows from Proposition~\ref{heapiso} that the heap $P_{\lambda, \nu}$ is embedded as a labeled order ideal of the heap $P_{\lambda, \mu}$.  Thus, we obtain an order-preserving map $\Lambda_{\lambda}^{\mu} \rightarrow J(P_{\lambda, \mu})$, the inverse to which is given by $\phi$. 
\end{proof}

We conclude this section by defining minuscule heaps and recalling one of their governing properties.

\begin{df}
Let $V$ be a minuscule representation of $\mathfrak{g}$ with minuscule weight $\lambda$.  The heap $P_{\lambda, w_0 \lambda}$, which we denote by $P_{\lambda}$, is the \textit{minuscule heap} for $V$.  
\end{df}

\begin{thm}[Rush--Shi \cite{Rush}, Theorem 6.3] \label{rushshi}
Let $V$ be a minuscule representation of $\mathfrak{g}$ with minuscule weight $\lambda$, and let $P_{\lambda}$ be the minuscule heap for $V$.  For all $1 \leq i \leq t$, let $P_{\lambda}^i \subset P_{\lambda}$ be the set of elements of $P_{\lambda}$ labeled by $i$, and let $t_i := \prod_{p \in P_{\lambda}^i} t_p$.  Then the following diagram is commutative.  
\[\renewcommand{\arraystretch}{1.0}
\begin{array}[c]{ccc}
J(P_{\lambda}) & \stackrel{\phi}{\rightarrow} & \Lambda_{\lambda} \\
\downarrow \scriptstyle{t_i} && \downarrow \scriptstyle{s_i} \\
J(P_{\lambda}) & \stackrel{\phi}{\rightarrow} & \Lambda_{\lambda}
\end{array}
\]
\end{thm}

\section{Proofs of Theorems 1.2 and 1.4}
In this section, we prove our main theorems.  We begin with three lemmas.  

\begin{lem} \label{plusminus}
Let $V$ be a minuscule representation of $\mathfrak{g}$ with minuscule weight $\lambda$ and minuscule heap $P_{\lambda}$.  Let $I \in J(P_{\lambda})$ be an order ideal.  Then $(\phi(I), \alpha_i^{\vee}) = 1$ if and only if $(\phi(\Psi(I)), \alpha_i^{\vee}) = -1$.  
\end{lem}

\begin{proof}
Suppose that $(\phi(I), \alpha_i^{\vee}) = 1$.  Then $s_i(\phi(I))$ covers $\phi(I)$ in $\Lambda_{\lambda}$, so, by Theorem~\ref{rushshi}, we see that $t_i(I)$ covers $I$ in $J(P_{\lambda})$.  It follows that there exists exactly one element $p \in P_{\lambda}^i$ such that toggling $I$ at $p$ yields $I \cup \lbrace p \rbrace$, and that toggling $I$ at any $q \in P_{\lambda}^i$ for which $q \neq p$ returns $I$.  

Therefore, toggling $\Psi(I)$ at $p$ yields $\Psi(I) \setminus \lbrace p \rbrace$, and toggling $\Psi(I)$ at any $q \in P_{\lambda}^i$ for which $q \neq p$ returns $\Psi(I)$.  It follows $\Psi(I)$ covers $t_i(\Psi(I))$ in $J(P_{\lambda})$.  We may conclude that $\phi(\Psi(I))$ covers $s_i(\phi(\Psi(I)))$ in $\Lambda_{\lambda}$, which implies that $(\phi(\Psi(I)), \alpha_i^{\vee}) = -1$, as desired. 
\end{proof}

\begin{lem} \label{invertible}
Let $\pi \colon \mathfrak{h}_{\mathbb{R}}^* \rightarrow \mathfrak{h}_{\mathbb{R}}^*$ be defined by $\theta \mapsto \sum_{i=1}^t (\theta, \alpha_i^{\vee}) \alpha_i^{\vee}$.  Then $\pi$ is an automorphism of $\mathfrak{h}_{\mathbb{R}}^*$.  
\end{lem}

\begin{proof}
Since $\pi$ maps the basis of fundamental weights $\lbrace \omega_1, \omega_2, \ldots, \omega_t \rbrace$ to the biorthogonal basis of simple coroots $\lbrace \alpha_1^{\vee}, \alpha_2^{\vee}, \ldots, \alpha_t^{\vee} \rbrace$, the desired result follows immediately.  
\end{proof}

\begin{lem} \label{ident}
The map $\mathfrak{h}_{\mathbb{R}}^* \rightarrow \mathfrak{h}_{\mathbb{R}}^*$ defined by $\theta \mapsto \sum_{i=1}^t (\theta, \alpha_i^{\vee}) \omega_i$ is the identity on $\mathfrak{h}_{\mathbb{R}}^*$.  
\end{lem}

\begin{proof}
The identity is the unique endomorphism of $\mathfrak{h}_{\mathbb{R}}^*$ that preserves each fundamental weight.  
\end{proof}

\subsection{Order Ideal Cardinality}

\begin{lem} \label{whatisordercard}
Let $V$ be a minuscule representation of $\mathfrak{g}$ with minuscule weight $\lambda$ and minuscule heap $P_{\lambda}$.  Let $f^i \colon J(P_{\lambda}) \rightarrow \mathbb{R}$ be defined by $I \mapsto |I \cap P_{\lambda}^i|$.  Then $f^i(I) = 2 \frac{(\lambda, \omega_i) - (\phi(I), \omega_i)}{(\alpha_i, \alpha_i)}$.   
\end{lem}

\begin{proof}
From Theorem~\ref{structure}, we see that 
\begin{align*}
\phi(I) & = \lambda - f^1(I) \alpha_1 - f^2(I) \alpha_2 - \cdots - f^t(I) \alpha_t \\& = \lambda - f^1(I) \alpha_1^{\vee} \frac{(\alpha_1, \alpha_1)}{2} - f^2(I) \alpha_2^{\vee} \frac{(\alpha_2, \alpha_2)}{2} - f^t(I) \alpha_t^{\vee} \frac{(\alpha_t, \alpha_t)}{2}. 
\end{align*}
Hence \[(\phi(I), \omega_i) = (\lambda_i, \omega_i) - f^i(I) \frac{(\alpha_i, \alpha_i)}{2},\] so \[f^i(I) = 2 \frac{(\lambda_i, \omega_i) - (\phi(I), \omega_i)}{(\alpha_i, \alpha_i)}.\] 
\end{proof}

We proceed to the proof of Theorem~\ref{mainorder}.  It follows from Lemma~\ref{invertible} that we may choose $\theta_1, \theta_2, \ldots, \theta_t \in \mathfrak{h}_{\mathbb{R}}^*$ so that $\pi(\theta_i) = 2 \frac{\omega_i}{(\alpha_i, \alpha_i)}$ for all $1 \leq j \leq t$.  Then \[2 \frac{(\phi(I), \omega_i)}{(\alpha_i, \alpha_i)} = (\phi(I), \pi(\theta_i)) = \sum_{j=1}^t (\phi(I), \alpha_j^{\vee})(\theta_i, \alpha_j^{\vee}).\] 

Let $m$ be the order of the Fon-Der-Flaass action $\Psi$ on $P_{\lambda}$.  Note that 
\begin{align*}
\sum_{k=0}^{m-1} 2 \frac{(\phi(\Psi^k(I)), \omega_i)}{(\alpha_i, \alpha_i)} &= \sum_{k=0}^{m-1} \sum_{j=1}^t (\phi(\Psi^k(I)), \alpha_j^{\vee}) (\theta_i, \alpha_j^{\vee}) \\ & = \sum_{j=1}^t (\theta_i, \alpha_j^{\vee}) \left( \sum_{k=0}^{m-1} (\phi(\Psi^k(I)), \alpha_j^{\vee}) \right) \\ & = 0,
\end{align*}

where the equality $\sum_{k=0}^{m-1} (\phi(\Psi^k(I)), \alpha_j^{\vee}) = 0$ follows from Lemma~\ref{plusminus}.  

It is now an immediate consequence of Lemma~\ref{whatisordercard} that $f^i$ is $c$-mesic with $c = 2 \frac{(\lambda, \omega_i)}{(\alpha_i, \alpha_i)}$.  \qed

\subsection{Antichain Cardinality}

\begin{lem}
Let $V$ be a minuscule representation of $\mathfrak{g}$ with minuscule weight $\lambda$ and minuscule heap $P_{\lambda}$.  Let $g^i \colon J(P_{\lambda}) \rightarrow \mathbb{R}$ be the cardinality of the set $\lbrace p \in P_{\lambda}^i : t_p(I) = I \setminus \lbrace p \rbrace \rbrace$.  Then \[\sum_{k=0}^{m-1} g^i(\Psi^k(I)) = \sum_{k=0}^{m-1} 2 \frac{(\phi(\Psi^k(I)), \alpha_i^{\vee}) (\phi(\Psi^k(I)), \omega_i)}{(\alpha_i, \alpha_i)}.\] 
\end{lem}

\begin{proof}
Lemma~\ref{plusminus} tells us that $(\phi(I), \alpha_i^{\vee}) = 1$ if and only if $(\phi(\Psi(I)), \alpha_i^{\vee}) = -1$.  If $(\phi(I), \alpha_i^{\vee}) = 1$, then \begin{align*} 2 \frac{(\phi(I), \alpha_i^{\vee}) (\phi(I), \omega_i)}{(\alpha_i, \alpha_i)} &+ 2 \frac{(\phi(\Psi(I), \alpha_i^{\vee})) (\phi(\Psi(I)), \omega_i)}{(\alpha_i, \alpha_i)} \\ &= 2 \frac{(\phi(I), \omega_i)}{(\alpha_i, \alpha_i)} - 2 \frac{(\phi(\Psi(I)), \omega_i)}{(\alpha_i, \alpha_i)} \\ &= 2 \frac{(\alpha_i, \omega_i)}{(\alpha_i, \alpha_i)} = 1 \end{align*} because the bases $\lbrace \alpha_1^{\vee}, \alpha_2^{\vee}, \ldots, \alpha_t^{\vee} \rbrace$ and $\lbrace \omega_1, \omega_2, \ldots, \omega_t \rbrace$ are biorthogonal.  

Hence \begin{align*} \sum_{k=0}^{m-1} 2 \frac{(\phi(\Psi^k(I)), \alpha_i^{\vee}) (\phi(\Psi^k(I)), \omega_i)}{(\alpha_i, \alpha_i)} &= \sum_{k=0}^{m-1} |\lbrace p \in P_{\lambda}^i : t_p(\Psi^k(I)) = I \cup \lbrace p \rbrace \rbrace| \\ &= \sum_{k=0}^{m-1} |\lbrace p \in P_{\lambda}^i : t_p(\Psi^k(I)) = I \setminus \lbrace p \rbrace \rbrace| \\ &= \sum_{k=0}^{m-1} g^i(\Psi^k(I)). \end{align*}
\end{proof}

We proceed to the proof of Theorem~\ref{mainanti}.  Suppose that $\mathfrak{g}$ is simply laced, and let $\Omega$ be the common length of the roots of $\mathfrak{g}$.  Let $g \colon J(P_{\lambda}) \rightarrow \mathbb{R}$ be the antichain cardinality statistic.  Then \begin{align*}\sum_{k=0}^{m-1} g(\Psi^k(I)) &= \sum_{k=0}^{m-1} \sum_{i=1}^t g^i(\Psi^k(I)) \\ &=  \sum_{k=0}^{m-1} \sum_{i=1}^t 2 \frac{(\phi(\Psi^k(I)), \alpha_i^{\vee})(\phi(\Psi^k(I)), \omega_i)}{\Omega^2} \\& = \sum_{k=0}^{m-1} \left( \Psi^k(I), \sum_{i=1}^t 2 \frac{(\phi(\Psi^k(I)), \alpha_i^{\vee}) \omega_i} {\Omega^2} \right) \\ &= \sum_{k=0}^{m-1} 2 \frac{(\Psi^k(I), \Psi^k(I))}{\Omega^2}, \end{align*} where the last equality follows from Lemma~\ref{ident}.  

Since the action of $W$ on $\Lambda_{\lambda}$ is transitive and $W \subset O(\mathfrak{h}_{\mathbb{R}}^*)$, it follows that $(\mu, \mu) = (\lambda, \lambda)$ for all $\mu \in \Lambda_{\lambda}$.  Hence $g$ is $c$-mesic with $c = 2 \frac{(\lambda, \lambda)}{\Omega^2}$.  \qed

\section{Acknowledgments}

The research that culminated in this article had its origins in a project undertaken at the Massachusetts Institute of Technology under the aegis of the Research Science Institute, administered by the Center for Excellence in Education.  It is the authors' pleasure to extend their gratitude to Pavel Etingof, Slava Gerovitch, David Jerison, and Tanya Khovanova for coordinating the RSI program within the MIT Department of Mathematics.  They also thank Alexander Postnikov, James Propp, Victor Reiner, Thomas Roby, XiaoLin Shi, and Nathan Williams for helpful conversations.  

The first author is presently supported by the US National Science Foundation Graduate Research Fellowship Program.

\end{document}